\documentclass[a4,12pt,reqno]{amsart}
\usepackage{amsmath}
\usepackage{amsthm}
\usepackage{graphicx}
\usepackage{amsfonts}
\usepackage{bm}
\usepackage{bbm}
\usepackage{textcomp}
\usepackage{amssymb}
\usepackage{mathrsfs}
\usepackage{enumerate}
\usepackage{hyperref}
\usepackage{geometry}
\usepackage{xcolor}

\numberwithin{equation}{section}

\newtheorem{theorem}{Theorem}[section]
\newtheorem{lemma}[theorem]{Lemma}
\newtheorem{proposition}[theorem]{Proposition}
\newtheorem{corollary}[theorem]{Corollary}

\theoremstyle{definition}

\newtheorem{definition}[theorem]{Definition}
\newtheorem{remark}[theorem]{Remark}
\newtheorem{asmp}[theorem]{Assumption}

\def\d{{\mathrm d}}
\def\E{{\mathbb E}}

\def\R{{\mathbb R}}

\def\P{{\mathcal P}}

\def\L{{\mathcal L}}

\title[Mean-field agency problem with accidents]{Agency problem and mean field system of agents with moral hazard, synergistic effects and accidents}
\author{Thibaut Mastrolia and Jiacheng Zhang}
\email{mastrolia@berkeley.edu}\email{jiachengz@berkeley.edu}
\address{Department of Industrial Engineering and Operations Research, UC Berkeley\newline 4141 Etcheverry Hall, Berkeley, CA 94720, United States}
\date{\today}

\begin{document}

\begin{abstract}
We investigate the existence of an optimal policy to monitor a mean field systems of agents managing a risky project under moral hazard with accidents modeled by L\'evy processes magnified by the law of the project. We provide a general method to find both a mean field equilibrium for the agents and the optimal compensation policy under general, sufficient and necessary assumptions on all the parameters. We formalize the problem as a bilevel optimization with the probabilistic version of a mean field games which can be reduced to a controlled McKean-Vlasov SDE with jumps. We apply our results to an optimal energy demand-response problem with a crowd of consumers subjected to powercut/shortage when the variability of the energy consumption is too high under endogenous or exogenous strains. In this example, we get explicit solution to the mean field game and to the McKean-Vlasov equation with jumps. \end{abstract}

\maketitle

\section{Introduction}

The agency problem, also known as Principal-Agent problem, emerged in economical and supply chain management's literature in 70's. It occurs when two entities with different (and sometimes opposite) interests committed with a contract to manage a possibly risky project. One entity, named the ``Principal'', delegates the decision of the other named the ``agent'' and monitor the actions of the latter. Both these entities act for their own interests with different information available. Such a situation leads to different type of hazards for all the contract's committed. In this paper, we will focus on three main kinds of hazards.

\begin{itemize}
\item \textbf{Moral hazard.} The economist Paul Krugman in \cite{krugman2009return} defined moral hazard as ``any situation in which one person makes the decision about how much risk to take, while someone else bears the cost if things go badly.'' It appears when there is information asymmetry, for example if the Principal does not observe the decision of the agent, and when the contract affects the behaviors of the parties. From an operations research and economical point of views, this problem is identified as a bilevel optimization and is reduced to find a Stackelberg equilibrium between the leader (the principal) and the follower (the agent). This kind of problem in continuous-time and modeling the uncertainty of the project's dynamic with a Brownian motion has been investigated in the pioneer article \cite{holmstrom1987aggregation}. We recall the Principal-Agent paradigm in the Brownian model as stated by Holstr\"om and Milgrom. Suppose that the agent chooses an action $\alpha$ controlling the drift of a Brownian motion $W$ with variance $\sigma^2$. Equivalently, the agent modifies the law of a primal Brownian motion by choosing a probability $\mathbb P^\alpha$ and by using Girsanov theorem. The accumulated profit $X$ of the Principal has the following dynamic
\[ \d X_t=\alpha_t\d t+\sigma \d W^\alpha_t,\; X_0\in \mathbb R,\] where $W^\alpha$ is a Brownian motion under $\mathbb P^\alpha$. At time $0$, the Principal proposes a compensation $\xi$ of the entire realized path of the $X$ given at a terminal time $T>0$. The agent chooses $\alpha_t$ at any time $t\leq T$ in full knowledge of the history $\{X_s,\; s\leq t\}$ subjected to a cost of effort $c(X,\alpha_t)$. The Principal's problem is to select a sharing rule $\xi$ and instructions $\alpha$ for the agent under two standard constraints that 
\begin{itemize}
\item Incentive compatibility (IC): the agent can maximize they expected utility by following instructions proposed by the Principal; 
\footnote{Note that in moral hazard the Principal does not observe the action of the agent and so cannot impose it to the later. However, incentive compatibility condition ensures that the recommendation proposed by the Principal is optimal for the agent when the compensation $\xi$ is fixed. Under several recommendation policies, the agent is indifferent and is supposed to follow the best recommendation proposed by the Principal. This situation is known as the second-best case, as opposed to the first-best problem in which the principal imposes a level of effort to the agent, without considering the constraint (IC).}
\item Reservation constraint (R): the agent can attain a certain minimum level of expected utility from the contract $R_0$.
\end{itemize}
 The principal problem, also known as the contracting problem is formally stated as
\begin{align}
\label{pb:HM}\sup_{\xi,\alpha}\;\;&\mathbb E^\alpha[U_P(\xi-X_T)]\; \\
\label{ICHM}\text{subject to: } &\alpha \in \arg\max_{\tilde \alpha} \mathbb E^{\tilde\alpha}\Big[U_A\Big(\xi-\int_0^T c(X,\alpha_t)\d t\Big)\Big]\\
\label{RHM}\text{and }&\mathbb E^\alpha\Big[U_A\Big(\xi-\int_0^T c(X,\alpha_t)\d t\Big)\Big]\geq R_0,
\end{align}
where $U_A$ and $U_P$ are utility functions of the agent and the Principal respectively. 
The optimization problem \eqref{pb:HM} corresponds to the Principal optimization as in \cite[(1)]{holmstrom1987aggregation} while constraints \eqref{ICHM} and \eqref{RHM} are respectively the incentive compatibility condition and the reservation utility constraints \cite[(2) and (3)]{holmstrom1987aggregation}. Note that Holmstrom and Milgrom does not impose any integrability condition on $\xi$ and $\alpha$ excepting those required to define the expectation considered. The problem \eqref{pb:HM} is a bilevel optimization since the problem of the agent \eqref{ICHM} is embedded in the contracting problem \eqref{pb:HM}. It is reduced to find a Stackelberg equilibrium when the Principal leads the game with the agent and has been solved for exponential utilities in \cite[Theorem 7]{holmstrom1987aggregation} emphasizes a particular form of the contract $\xi$ given by \cite[Theorem 6, (23)]{holmstrom1987aggregation}. This problem has been investigated by the mathematical community in the last fifteen years. A particular extension to random horizon has been studied in \cite{sannikov2008continuous,possamai2020there}. The recent article \cite{cvitanic2018dynamic} has proposed a comprehensive and rigorous mathematical method to solve this problem under integrability assumption for the contract $\xi$, in the Brownian model with controlled drift and volatility by using the theory of second-order backward stochastic differential equations to solve (IC) and classical verification result for solving \eqref{pb:HM}.

\item \textbf{Synergistic effects.} Synergy refers to an interaction of entities leading to a greater impact on a whole than the simple sum of its parts. In the nature, this concept is ubiqious, for example the combination of atoms to create molecule in chemistry or in geology with the stone synergy. In human society, synergistic effects can be either beneficial or unfavorable. For instance, in medical science, a synergistic effect refers to several hazards having a greater effect on the level of risk they pose to worker health and safety than the severity of the combination of these hazards. In corporate science, synergy refers to the opportunity of a union of corporate entities to reduce, or eliminate expenses, see for example \cite{chatterjee1986types}. In finance and referring to \cite{smith1994combined}, synergy can be used for cash slack. Mathematically, we model this effect by adding a dependancy with respect to the realization of $X$ in the drift and the volatility of the accumulated profit of the Principal so that 
\[ \d X_t = b(t,\alpha_t,X_t)\d t+\sigma(t,X_t)\d W_t^\alpha.\]
When the volatility is also controlled by the agent, this framework coincides with \cite{sannikov2008continuous,cvitanic2018dynamic}. 
\item \textbf{Accidents.} We call ``accidents'' some jumping risks affecting negatively the dynamic of $X$: for example an electrical breakdown caused by a huge strain on the electric grid.  Accidents' prevention has been investigated in \cite{hartman2012optimal} for mortgages subjected to default risk, in \cite{sung1997corporate} for corporate insurance and in \cite{capponi2015dynamic} by using compound Poisson processes to model accidents. More recently, \cite{martin2021bsde,martin2021class} have investigated this problem by mixing a Brownian motion and a single jump process to model the project's profitability subjected to shutdown. Extended the framework of these paper, the dynamic of the project $X$ becomes
\[ \d X_t = b(t,\alpha_t,X_t)\d t+\sigma(t,X_t)\d W_t^\alpha - \d J_t,\]
for some jump process $J$ with compensator $\lambda$ depending possibly on the effort $\alpha$ of the agent and the realization of the process $X$ itself (synergistic effect). 
\end{itemize}

The case of several (but finite) number of agents has been studied in \cite{goukasian2010optimal,kang2013nash,elie2019contracting,mastrolia2017moral}. In these articles, the authors solve the $N-$player version of \eqref{ICHM} by finding a Nash equilibrium and then provide the optimal contract proposed by the Principal. When the number of agents goes to $+\infty$, the condition \eqref{ICHM} is reduced to find a mean field equilibrium introduced by Lasry and Lions in \cite{lasry2007mean,lasry2006jeux,lasryII2006jeux} and Huang, Caines and Malham\'e \cite{huang2006large,huang2007large} independently. In these papers, the authors use PDE method to extend the definition of a Nash equilibrium in optimal control problems when the number of players goes to $+\infty$. We refer to \cite{cardaliaguet2010notes,gueant2009reference} for pedagogical reviews of this method. A probabilistic method to find a mean field equilibrium has been investigated by Carmona and Lacker in \cite{carmona2015probabilistic}. We also refer to the books \cite{carmona2018probabilistic} for more details. This method is particularly suitable for principal-agent problem with moral hazard, since it deals with weak formulation of a stochastic control problem. Using this formulation, an extension of principal-agent problem with a mean field systems of interacting agents under moral hazard and synergy effect has been introduced in \cite{elie2019tale} when the accumulated profit of the Principal $X$ has no accidents with dynamic given by 
\[ \d X_t= b(t,\alpha_t,X_t,\mu_t)\d t+\sigma(t,X_t)\d W_t^\alpha, \]
where $\mu_t$ denotes some probability distribution. It has then been applied for energy optimal demand-response in \cite{elie2021mean,campbell2021deep}, see the dedicated paragraph below for more details. Mathematically, solving
\eqref{pb:HM} is split into two steps using two different stochastic control tools. First and following \cite{carmona2015probabilistic}, a mean field equilibrium is a pair $(\alpha^\star,\mu^\star)$ satisfying \eqref{ICHM} and the fixed point condition $\mu^\star_t=\mathbb P^{\alpha^\star,\mu^\star}\circ X_t^{-1}$. Secondly, the solution of the contracting problem \eqref{pb:HM} is reduced to a stochastic control problem of a controlled McKean-Vlasov SDE driven by a Brownian motion. This can be solved by using a verification result on the space of measure, see among others \cite{carmona2013control,bensoussan2013mean,pham2017dynamic,wu2020viscosity}.\\

We now turn to the main contributions together with the general structure of our article.\\

\textbf{Model contribution: Principal-mean field systems of Agents with accidents.}
One of the main contributions of this article is to consider the management of a stochastic process with both continuous part and discontinuous part with mean field interactions. The model is described in Section \ref{section:model}. Mathematically, the accumulated profit and loss of the Principal has the following dynamic:
\begin{equation}\label{X:intro}
\d X_t=b(t,X_t,\mu_t,\alpha_t)\d t+\sigma(t,X_t)\d W^{\alpha,\mu}_t + \int_{\zeta\neq 0} \zeta N(\d t,\d \zeta),\end{equation}
where $N$ is a L\'evy process with compensator kernel $\lambda_t^{\alpha,\mu}$ depending on the effort of a mean field system of agents $\alpha$, the process itself $X$ and its law at time $t$ denoted by $\mu_t$, modelling ``accidents'' reducing the profit (or good events increasing it depending on the distribution). The Principal thus benefits from the accumulated profit at time $T$, on some function $\pi$ of this profit reducing by the compensation $\xi$ given to the agents and subjected to additional continuous cost $g$ and accidents' cost $\ell$. The contracting problem is given below and mathematically formalized in \eqref{bilevel} 
\[V^P_0:=\sup_{(\mu,\alpha,\xi)}\; \mathbb E^{\alpha,\mu}\Big[ U_P\Big(\pi(X_T) -\xi - \int_0^T \int_{\zeta\neq 0}\ell(t,X_t,\zeta,\mu_t)N(\d t ,\d \zeta)-\int_0^T g(t,X_t,\mu_t)\d t\Big)\Big],\]
subject to 
\begin{itemize}
\item incentive compatibility $($IC$)$, where the value function of the agents is given by the expected value of an exponential utility of the compensation $\xi$ reduced by a cost of effort $c_0$ wit additional incomes/penalty given by accidents occurrences $c_1$ together with the existence of an equilibrium for the mean field system of agents characterized by a fixed point condition $\mu^\star_t=\mathbb P^{\alpha^\star,\mu^\star}\circ X_t^{-1}$ denoted by $($MFG$)$ defined in \eqref{MF};
\item a reservation utility constraint $($R$)$;
\item technical assumptions denoted by $(\mathcal I_\xi),(\mathcal I_-)$ and $(\mathcal I_\alpha)$. 
\end{itemize}

In Section \ref{section:bilevel}, we focus on solving the bilevel programming mean field version of \eqref{pb:HM} when $X$ is solution to \eqref{X:intro}. We first state the existence and the characterization of a mean field equilibrium in this framework, see Section \ref{section:MFG}. The probabilistic formulation of a mean field equilibrium with jumps component in the controlled process has been developed in \cite{benazzoli2020mean}. By using this formulation, we first find a pair $(\alpha^\star,\mu^\star)$ satisfying $($IC$)$ and $($MF$)$ in Theorem \ref{thm:ICC} and its Corollary \ref{mfgcharacterization} by providing a characterization of a mean field equilibrium with a system of controlled McKean-Vlasov equations. Then, we solve the bilevel programming in Section \ref{section:MFPDE} and Theorem \ref{thm:verif} by using verification results for stochastic control problem of McKean-Vlasov SDE with jumps introduced in \cite{burzoni2020viscosity,guo2020s}.\\

\textbf{Solving agency problems with mean field interactions and jumps: necessary and sufficient conditions.} So far, the existing articles have solved agency problems in continuous time under strong assumption on the integrability of both the contract $\xi$ proposed by the Principal, the integrability of the effort of the agent $\alpha$ and the cost induced by it. In this paper we unify the conditions imposed in the existing literature with a general one sufficient to solve our problem and necessary for the well-posedness of the quantities studied, see Remark \ref{remark:fundamental} below for more details. More precisely, unlike \cite{elie2019tale} or \cite{campbell2021deep}, we do not impose exponential integrability of any orders or convexity with respect to the process $X$ for the class of contract. We prove that weaker conditions are enough to solve the problem, see Condition $(\mathcal I_\xi)$ below. Compared to \cite{cvitanic2018dynamic} we impose neither boundedness of optimal control $\alpha$ nor integrability condition on the cost function considered. Instead, we unify this conditions into general constraints necessary again for the well-posedness of the expectations considered, see $(\mathcal I_\alpha)$ and $(\mathcal I_-)$.\\

\textbf{Application in energy optimal demand-response with powercut or shortage.} Agency problem with moral hazard and mean field systems occurs especially in electricity optimal demand response. In this kind of problem, a producer (the Principal) aims at designing an optimal electricity price policy to incentivize a crowd of consumers to manage their energy consumption sustainably. This problem has been first mathematically introduced in continuous time in \cite{aid2022optimal} for one consumer monitored by an energy producer. It has then been extended in \cite{elie2021mean} to a mean field systems of identical consumers. In particular, the authors prove that the energy consumption of the mean field systems is a fundamental factor to implement optimally a contract, improving substantially the energy management by reducing the cost for the producer. More recently, \cite{campbell2021deep} has developed neural networks method to solve this kind of problem. The purpose of Section \ref{section:electricity} is to extend these results to energy consumption (for example gas, gasoline, water, electricity) by adding possible breakdown when the strain on an electricity grid or the energy demand becomes too strong. We give in particular closed form solutions to the problem with an explicit optimal contract to monitor energy demand-response subjected to blackout in Proposition \ref{prop:example}.

\section{Principal/mean-field Agent model and Stackelberg game}\label{section:model}

\indent \quad This section is dedicated to specify Principal-mean field systems of Agents with accidents. The regime considered is the usual mean-field setting where we consider an entire crowd of Agents hired by one Principal and we focus on a representative one interacting with the theoretical distribution of the infinite number of other players.

\quad We fix a terminal time $T>0$. Let $(\Omega, {\mathcal F}, \mathbb P)$ be a probability space endowed with a Brownian motion $W$ and a Poisson random measure $N$ on $[0,T]\times \mathbb R$ with predictable intensity kernel  $\lambda^0_t(\d \zeta)$. We set
\begin{equation}\label{sde:primal}
\d X_t=\sigma(t,X_t)\d W_t + \int_{\zeta\neq 0}\zeta N(\d t,\d \zeta),\; X_0=x\in \mathbb R.
 \end{equation}

We define $\mathbb F$ the filtration generated by $X$. We denote by $\mathcal V_T$ the set of real random variable $\mathcal F_T-$measurable.
 \begin{asmp}
 The processes $\sigma(\cdot,X_\cdot)$ and the kernel $\lambda^0$ are chosen such the SDE \eqref{sde:primal} has a unique solution. 
\end{asmp}

From now on, we denote by $\tilde p$ the classical conjugate of any $p>1$ defined by 
\[ \frac1p+\frac1{\tilde p}=1\; \Longleftrightarrow\; \tilde p =\frac{p}{p-1}.\]
\subsection{Agents control strategy} We assume that the agents modifies the law of both the continuous and jumping parts of $X$ with an $\mathbb R^n$-vector of actions, $n\geq 1$. We set
\begin{itemize}
\item $\mathcal P_2^n$ denotes the set of probability measures on $\mathbb R^n$ with $n\geq 1$. When $n=1$ we omit the index $n$ and we write $\mathcal P_2$ for $\mathcal P_2^1$;
\item $\mathcal P_2^n([0,T])$ is the set of function from $[0,T]$ into $\mathcal P_2^n$;
\item $b:[0,T]\times \mathbb R\times \mathcal P_2^1\times \mathbb R^d\longrightarrow \mathbb R$ such that $b:t\in [0,T] \longmapsto b(t,\cdot)$ is predictable;
\item $K:[0,T]\times \mathbb R\times \mathbb R\times \mathcal P_2\times \mathbb R^d\longrightarrow (0,\infty)$ such that $K(t,\cdot)$ is predictable;
\item a Doleans Dade exponential process $\mathcal E^{\alpha,\mu}$ defined by
\begin{align*}
\mathcal E^{\alpha,\mu}_t&:=\exp\Big( \int_0^t \frac{b(s,X_s,\mu_s,\alpha_s)}{\sigma(s,X_s)}\d W_s-\frac12  \int_0^t \Big|\frac{b(s,X_s,\mu_s,\alpha_s)}{\sigma(s,X_s)}\Big|^2\d s\\
&\hspace{3em} +\int_0^t \int_{\zeta\neq 0} \ln\big(K(s,X_s,\zeta,\mu_s,\alpha_s)\big) N(\d s,\d \zeta)\\
&\hspace{3em}  + \int_0^t \int_{\zeta\neq 0} \Big[1- K(s,X_s,\zeta,\mu_s,\alpha_s)\big] \lambda^0_s(\d \zeta)\d s \Big).
\end{align*}
\end{itemize}

\noindent For any $\mu\in \mathcal P_2([0,T])$, we define $\widetilde{\mathcal A}^\mu$ by
\[\widetilde{\mathcal A}^\mu:=\{\alpha:[0,T]\times \Omega\longrightarrow \mathbb R^d \text{ s.t. } \mathbb E[\mathcal E^{\alpha,\mu}_T]=1 \}.\]

\noindent We deduce from Girsanov Theorem (see for example \cite[Theorem 1.35]{oksendal2019stochastic}) that there exists a probability measure $\mathbb P^{\alpha,\mu}$ with $\alpha\in \widetilde{\mathcal A}^\mu$ and $\mu_t\in \mathcal P_2$ for any $t\in [0,T]$ such that 
 $W^{\alpha,\mu}$ and $\tilde N^{\alpha,\mu}$ defined by
\[ W^{\alpha,\nu}:= W-\int_0^\cdot  \frac{b(t,X_t,\mu_t,\alpha_t)}{\sigma(t,X_t)} \d t,\quad \tilde N^{\alpha,\mu}(\d t,\d \zeta):=N(\d t,\d \zeta)-K(t,X_t,\zeta,\mu_t,\alpha_t)\lambda^0_t(\d \zeta)\d t.\]

\noindent are respectively a Brownian motion and a compensated jump measure under $\mathbb P^{\alpha,\mu}$. {The dynamic of the solution $X$ to the SDE \eqref{X:intro} under $\mathbb P^{\alpha,\mu}$ is given by}
\begin{align}\label{sde:controlled}
\d X_t&=b(t,X_t,\mu_t,\alpha_t)\d t+\sigma(t,X_t)\d W^{\alpha,\mu}_t\\
\nonumber& + \int_{\zeta\neq 0} \zeta\tilde N^{\alpha,\mu}(\d t,\d \zeta) + \int_{\zeta\neq 0} \zeta K(t,X_t,\zeta,\mu_t,\alpha_t)\lambda^0_t(\d \zeta)\d t, \end{align}

\subsection{The Stackelberg mean-field game}

For a fixed compensation $\xi\in \mathcal V_T$, the problem the system of agents is to find $(\alpha,\mu)\in \widetilde{\mathcal A}^\mu \times \P_2([0,T])$ such that
\begin{equation}\text{MFG}(\xi):\label{MF} \begin{cases}
&V_0(\mu,\xi):=\underset{\tilde\alpha\in \widetilde{\mathcal A^{\mu}}}{\sup} v_0(\tilde\alpha;\mu,\xi) = v_0(\alpha;\mu,\xi)\\
&\mu_t=\mathbb P^{\alpha,\mu}\circ X_t^{-1}
\end{cases}
\end{equation}
with
\[v_0(\alpha;\mu,\xi):=\mathbb E^{\alpha,\mu}\Big[U_A\Big(\xi-\int_0^T c_0(t,X_t,\mu_t,\alpha_t) \d t+\int_0^T \int_{\zeta\neq 0}c_1(t,X_t,\zeta,\mu_t,\alpha_t)N(\d t,\d \zeta) \Big) \Big],\]
where
\begin{itemize}
\item $U_A(x):=-e^{-\gamma x}$ where $\gamma>0$ is the risk aversion parameter of the mean field of agents;
\item $c_0$ denotes a cost induced by the action of the agent while $c_1$ is a reward depending on the jumps of $X$.
\item We denote MFG$(\xi)$ as the set of pairs $(\alpha,\mu)$ satisfying \eqref{MF}
\end{itemize}
The Stackelberg game can be thus written as a bi-level mean field optimization problem under constraints:
\begin{equation}\label{bilevel}
  \begin{aligned}
  V^P_0:=\underset{ \underset{(\alpha,\mu)\in\widetilde{\mathcal A}^\mu\times\mathcal P_2([0,T])}{ \xi\in \mathcal V_T}}{\sup}\; \mathbb E^{\alpha,\mu}\Big[ U_P\Big(&\pi(X_T) -\xi - \int_0^T \int_{\zeta\neq 0}\ell(t,X_t,\zeta,\mu_t)N(\d t ,\d \zeta)
  \\
  &-\int_0^T g(t,X_t,\mu_t)\d t\Big)\Big],
  \end{aligned}
\end{equation}
subject to 
\begin{align}
\nonumber\text{($\mathcal I_\xi$): }&\mathbb E^{\alpha,\mu}[e^{-\gamma'\xi}]+ \mathbb E^{\alpha,\mu}[e^{\eta'\xi}]<\infty, \text{ for some } \gamma'>\gamma \text{ and }\eta'>\eta \text{ uniformly in $(\alpha,\mu)$},\\
 \nonumber \text{($\mathcal I_-$): }& \int_t^T \bigg\{\int_{\zeta\neq 0}c_1(s,X_s,\zeta,\mu_s,\alpha_s)N(\d s,\d \zeta)-c_0(s,X_s,\mu_s,\alpha_s) \bigg\}\d s< \infty,\; \forall\, t\in [0,T],\; \mathbb P\text{ -a.s.},\\
 \nonumber  \text{($\mathcal I_\alpha$): }&\text{(i).}\; \mathbb E^{\alpha,\mu}[e^{-\gamma \kappa\big(-\int_0^t c_0(s,X_s,\mu_s,\alpha_s) ds+\int_0^t \int_{\zeta\neq 0}c_1(s,X_s,\zeta,\mu_s,\alpha_s)N(\d s,\d \zeta)\big)}]<\infty,\\
  \nonumber & \quad\;\text{ for some } \kappa>\frac{\gamma'}{\gamma'-\gamma}  \text{ for any $t\in [0,T]$;}\\
 \nonumber &\text{(ii).}\; \text{there exists }  \varepsilon<\varepsilon^*:=\min\big(\frac{\gamma'-\gamma}{\gamma},\frac{\kappa(\gamma'-\gamma)-\gamma'}{\gamma'+\kappa\gamma}\big)\\
 \nonumber &\quad\;\text{ such that }\mathbb E^{\alpha,\mu}[\sup_{t\in [0,T]} e^{\gamma \tilde\rho \big(-\int_0^t c_0(s,X_s,\mu_s,\alpha_s) ds+\int_0^t\int_{\zeta\neq 0}c_1(s,X_s,\zeta,\mu_s,\alpha_s)N(\d s,\d \zeta)\big)}]<\infty,
 \\
 \nonumber &\qquad\text{ with } \tilde\rho=\frac{1+\varepsilon^*}{\varepsilon^*-\varepsilon}>1;\\
 \nonumber \text{(R): }& V_0(\mu,\xi)\geq R_0,\\
\nonumber \text{(IC): }& V_0(\mu,\xi)= v_0(\alpha;\mu,\xi),\\
\nonumber \text{(MF): } & \mathbb P^{\alpha,\mu}\circ X_t^{-1}=\mu_t,
\end{align}
where 
\begin{itemize}
\item $U_P(x):=-e^{-\eta x},$ is an exponential utility function with risk aversion parameter $\eta>0$; 
\item $\pi:\mathbb R\longrightarrow \mathbb R$ is a profit function associated with the final value of $X$;
\item $\ell:\mathbb R\longrightarrow \mathbb R$ is a loss function depending on the hazards occurrences $\Delta X$;
\item $g:\mathbb R\longrightarrow \mathbb R$ is a running cost; 
\item $(\mathcal I_\xi)$ is an integrability condition on $\xi\in \mathcal V_T$;
\item $(\mathcal I_\alpha)$ and $(\mathcal I_-)$ denotes integrability conditions on either the set of admissible control $\widetilde{\mathcal A}^\mu$ or the functions $c_0$ and $c_1$;
\item (R) denotes the reservation utility constraint for some $R_0\in \mathbb R$ fixed, 
\item (IC) is the incentive compatibility condition,
\item (MF) is the mean field condition. Note that (IC) combined with (MF) corresponds to the probabilistic version of a mean field game (see \cite{carmona2015probabilistic}). In particular, MFG$(\xi)$ the set of solution $(\alpha,\mu)$ to (IC) and (MF) when $\xi$ is fixed.
\end{itemize}

We finally define $\mathcal C$ as the set of admissible contracts
\[\mathcal C:=\{\xi\in \mathcal V_T \text{ such that } (\mathcal I_\xi) \text{ and  (R)} \text{ are satisfied and MFG}(\xi)\neq \emptyset \},\]
and 
\[\mathcal A^\mu:=\{\alpha\in \widetilde{\mathcal A}^\mu \text{ such that } (\mathcal I_\alpha) \text{ is satisfied} \}.\]

\begin{remark}\label{remark:fundamental}The problem \eqref{bilevel} is written under general technical constraints on both the control set and the model parameters. Although we are not able to prove rigorously that these constraints are not only sufficient but necessary, we explain in this remark why there are natural for the wellposedness of the problem. 
\begin{itemize}
\item[-]Constraint $(\mathcal I_\xi)$ requires an integrability of order $\gamma+\varepsilon$ rather than $\gamma$. This condition is natural to ensure that the values of both the agents and the principal are finite by using H\"older Inequality, it generalizes in particular \cite{elie2019tale} which required exponential moments of any orders and \cite{campbell2021deep} without any convex and regularity assumptions on the shape of the contract;
\item[-]Constraint $(\mathcal I_-)$ ensures that the value of the agents is not degenerated when $\xi=0$ at any time $t$, that is $v_t(\mu,0)<\infty$;
\item[-]Constraint $(\mathcal I_\alpha)$ is a technical condition set to find the optimizer in Condition (IC) when $\mu,\xi$ are fixed. This condition is weaker than the existing conditions in the literature, even for the one agent-principal case. It is satisfied for either bounded controls (see for instance \cite{euch2021optimal,elie2021mean}) or integrability/bounded conditions for the costs function (see \cite{elie2019tale}); 
\item[-] Constraints (R),(IC) and (MF) come from the principal-mean field system of agents framework. These conditions ensure that the principal cannot penalize too strongly on the agents (R); the principal proposes recommendations to the agents ensuring them the optimal best-reaction actions given a fixed compensation (IC); there exists a mean field equilibrium for the interacting system of agents (MF).
\end{itemize}
\end{remark}
\section{Solving the bi-level mean field optimization problem}\label{section:bilevel}

\subsection{MFG and McKean Vlasov SDE}\label{section:MFG}

We first focus on a smooth characterization of the constraints (R) and MFG$(\xi)$ when $\xi\in \mathcal C$ is fixed.\\

\noindent For all $\mathbb F-$predictable stopping time $\tau$ with value in $[0,T]$ and $(\mu,\alpha)\in \mathcal P_2([0,T])\times \widetilde{\mathcal A}^\mu$ we define\footnote{Here $\mathcal A^\mu_\tau$ denotes the restriction of $\mathcal A^\mu$ to the control $\alpha$ defined on $[\tau,T]$.}

\begin{equation*}
  \begin{aligned}
v_\tau(\alpha;\mu,\xi):=\mathbb E^{\alpha,\mu}\bigg[-\exp\bigg(-\gamma\Big(\xi&+\int_\tau^T \int_{\zeta\neq 0}c_1(s,X_s,\zeta,\mu_s,\alpha_s)N(\d s,\d \zeta)  
\\
&-\int_\tau^T c_0(s,X_s,\mu_s,\alpha_s) \d s\Big)\bigg)\Big|\mathcal F_\tau\bigg],
  \end{aligned}
\end{equation*}
and
\[V_\tau(\mu,\xi)=\underset{\alpha \in \widetilde{\mathcal A}^\mu}{\text{ ess sup }} v_\tau(\alpha;\mu,\xi).\]

\begin{lemma}[Dynamic programming principle]\label{DPP}
Let $t\in [0,T]$, for any $\mathbb F-$predictable stopping time $\tau$ with value in $[t,T]$, the dynamic programming principle is given by
\[V_t(\mu,\xi)= \underset{\alpha\in  \widetilde{\mathcal A}^\mu}{\mathrm{ ess\, sup\; }} \mathbb E_t^{\alpha,\mu}\Big[e^{-\gamma\big(-\int_t^\tau c_0(s,X_s,\mu_s,\alpha_s) ds+\int_t^\tau   \int_{\zeta\neq 0}c_1(s,X_s,\zeta,\mu_s,\alpha_s)N(\d s,\d \zeta)\big)}V_\tau(\mu,\xi)\Big],\;  \]
for $(\mu,\xi)\in \mathcal P_2([0,T])\times \mathcal C.$
\end{lemma}

\noindent As result of \cite{kunita2004representation} (see also \cite[Proposition 11.2.8.1]{jeanblanc2009mathematical}) we have the following result
\begin{lemma}[Martingale representation]\label{MRT}

For any $\mathbb F-$martingale $M$ there exists predictable processes with $Z:[0,T]\times \Omega \longrightarrow \mathbb R$ and $U:[0,T]\times\mathbb R\setminus\{0\}\times \Omega\longrightarrow\mathbb R$ such that
\[ \int_0^T |Z_s\sigma(s,X_s)|^2ds + \int_0^T \int_{|\zeta|>1}|U_s(\zeta)|\lambda_s^0(\d\zeta)\d s <\infty, \text{ a.s.},\]
and
$$ M_t = M_0 + \int_0^t Z_s dW^{\alpha,\mu}_s + \int_0^t  \int_{\zeta\neq 0}U_s(\zeta) \widetilde N^{\alpha,\mu}(\d s,\d \zeta). $$
\end{lemma}

\noindent {We denote by $\mathcal G(\mathbb R)$ the set of functions from $\mathbb R$ into $\mathbb R$.} Now we define for any $(t,x,z,u,\mu,\alpha)\in [0,T]\times \mathbb R\times \mathbb R\times \mathcal G(\mathbb R)\times \mathcal P_2\times\mathbb R^d$,

\begin{equation*}
  \begin{aligned}
 h(t,x,z,u,\mu,\alpha):=& \int_{\zeta\neq 0} \frac1{\gamma }\Big(1-e^{-\gamma  (u(\zeta)+c_1(t,x,\zeta,\mu,\alpha))}\Big) K(t,x,\zeta,\mu,\alpha)\lambda_t^0(\d \zeta)\\ 
&+b(t,x,\mu,\alpha)z-\frac12\gamma\sigma^2(t,x)z^2-c_0(t,x,\mu,\alpha),
  \end{aligned}
\end{equation*} 

\noindent and define $H(t,x,z,u,\mu):=\sup_\alpha h(t,x,z,u,\mu,\alpha)$. We denote by $\hat\alpha(t,x,z,u,\mu)$ a maximizer of $H$ with $t,x,z,u,\mu$ fixed. Let $\hat{\mathcal A}^\mu$ the set of processes $\hat\alpha$ such that $\hat\alpha_t:=\hat\alpha(t,X_t,Z_t,{U_t},\mu_t)$ for any $t$. 
\begin{asmp}
For any $(t,x,z,u,\mu)\in [0,T]\times \mathbb R\times \mathbb R\times \mathcal G(\mathbb R)\times \mathcal P_2([0,T])$
\[ \hat{\mathcal A}(t,x,z,u,\mu) \neq \emptyset \text{  and  } |\hat{\mathcal A}(t,x,z,u,\mu) |<\infty.\]
\end{asmp}

\begin{definition}\label{def:Q}
We  denote by $\mathcal Q$ the set of pair of predictable processes $(Z,U)$ with $Z:[0,T]\times \Omega \longrightarrow \mathbb R$ and $U:[0,T]\times\mathbb R\setminus\{0\}\times \Omega\longrightarrow\mathbb R$ such that
\begin{itemize}
\item[(i).] $\int_0^T |Z_s\sigma(s,X_s)|^2ds + \int_0^T \int_{|\zeta|>1}|U_s(\zeta)|\nu(\d\zeta)\d s+ \int_0^T \int_{|\zeta|\leq 1}|U_s(\zeta)|^2\nu(\d\zeta)\d s <\infty, \text{ a.s}.$
\item[(ii).] The integrability condition $(\mathcal I_\xi)$ is satisfied for $Y_0\in \mathbb R,$ and $\xi=Y_T^{Y_0,Z,U}$ where
\begin{align*}
Y_t^{Y_0,Z,U}:=Y_0+\int_0^t Z_s\sigma(s,X_s) d W_s+ \int_0^t \int_{\zeta\neq 0}U_s(\zeta)  N(\d s,\d \zeta)-\int_0^t H(s,X_s,Z_s,U_s,\mu_s)\d s.
\end{align*}

\item[(iii).] There exists $\tilde\gamma>\gamma$ such that

\begin{equation*} 
\sup_{(\mu,\alpha)\in \mathcal P_2\times \widetilde{\mathcal A}^\mu} ~ \mathbb E^{\alpha,\mu}[\sup_{t \in [0,T]} \; e^{-\tilde\gamma Y_t^{Y_0,Z,U}}] < \infty.\end{equation*}
\end{itemize}
\end{definition}
\begin{theorem}[Incentive compatibility condition and admissible contract]\label{thm:ICC}
Let $\xi\in \mathcal C$, then
\begin{itemize}
\item[1).] There exists a unique triplet $(Y_0,Z,U)\in (-\infty,\hat R_0]\times \mathcal Q$ such that \begin{align*}
\xi=Y_T^{Y_0,Z,U}&=Y_0+\int_0^T Z_s\d X^c_s+ \sum_{s\leq T} U_s(\Delta X_s)-\int_0^T H(s,X_s,Z_s,U_s,\mu_s)\d s,
\end{align*}
where $X^c$ is the continuous part of $X$ and $\Delta X$ is the jumping part.
\item[2).] $V_0(\mu,\xi)=-e^{-\gamma Y_0}$ and $\hat\alpha\in \hat{\mathcal A}^\mu$ is optimal with respect to \text{(IC)} for $\mu$ satisfying (MF) with $\alpha=\hat\alpha(\cdot,X_\cdot,Z_\cdot,U_\cdot,\mu_\cdot)$.
\end{itemize}
\end{theorem}

\begin{proof}
The proof is divided in several steps and extend the proof of \cite[Theorem 3.1]{euch2021optimal} to L\'evy processes. \\

\noindent \textit{Step 1. Doob-Meyer decomposition.} 
We define \[\widetilde Y_t(\alpha,\mu,\xi):=V_t(\mu,\xi)e^{-\gamma \big(-\int_0^t c_0(s,X_s,\mu_s,\alpha_s) ds+\int_0^t\int_{\zeta\neq 0}c_1(s,X_s,\zeta,\mu_s,\alpha_s)N(\d s,\d \zeta)\big)},\; t\in [0,T], \]
Note that
\begin{align*}
\sup_{(\alpha,\mu)}\mathbb E^{\alpha,\mu}[|\widetilde Y_t(\alpha,\mu,\xi)|]&= \sup_{\alpha,\mu}\mathbb E^{\alpha,\mu}[|V_t(\mu,\xi)| e^{-\gamma  \big(-\int_0^t c_0(s,X_s,\mu_s,\alpha_s) ds+\int_0^t\int_{\zeta\neq 0}c_1(s,X_s,\zeta,\mu_s,\alpha_s)N(\d s,\d \zeta)\big)}]\\
&= \sup_{\alpha,\mu}\mathbb E^{\alpha,\mu}[e^{-\gamma\xi}e^{-\gamma\big(-\int_0^T c_0(s,X_s,\mu_s,\alpha_s) \d s+\int_0^T \int_{\zeta\neq 0}c_1(s,X_s,\zeta,\mu_s,\alpha_s)N(\d s,\d \zeta)  \big)}].\end{align*}

By Young inequality, $(\mathcal I_\alpha)(i)$ and $(\mathcal I_\xi)$ there exists some constant $C_\kappa$ such that
\begin{align*}&\sup_{(\alpha,\mu)}\mathbb E^{\alpha,\mu}[|\widetilde Y_t(\alpha,\mu,\xi)|]\\
&\leq    C_\kappa( \sup_{\alpha,\mu}\mathbb E^{\alpha,\mu}[e^{-\gamma \kappa \big(-\int_0^T c_0(s,X_s,\mu_s,\alpha_s) ds+\frac1{\tilde\kappa} \int_0^T\int_{\zeta\neq 0}c_1(s,X_s,\zeta,\mu_s,\alpha_s)N(\d s,\d \zeta)\big)}] + \sup_{\alpha,\mu}\mathbb E^{\alpha,\mu}[e^{-\gamma \tilde\kappa \xi}])\\
&<\infty
\end{align*}

Hence, from Lemma \eqref{DPP}, we deduce that $\widetilde Y(\alpha,\mu,\xi)$ defined a $\mathbb P^{\alpha,\mu}-$supermartingale for any $(\mu,\alpha)\in \text{MFG}(\xi)$. Up to a selection of a c\`adl\`ag version of $\widetilde Y$, we have the following Doob-Meyer decomposition
\[ \widetilde Y_t(\alpha,\mu,\xi)= M_t^{\alpha,\mu}-A^{\alpha,\mu ; c}_t-A_t^{\alpha,\mu; d}, \]

\noindent where $M^{\alpha,\mu}$ is a $\mathbb P^{\alpha,\mu}-$martingale, $A^{\alpha,\mu}:=A^{\alpha,\mu; c}+A^{\alpha,\mu; d}$ is an integrable non-decreasing predictable process such that $A^{\alpha,\mu; c}_0=0$ and $A^{\alpha,\mu; d}_0=0$, with a pathwise continuous component $A^{\alpha,\mu;c}$ and a piecewise constant predictable process $A^{\alpha,\mu; d}$. From the martingale representation theorem given by Lemma \ref{MRT}, there exists two predictable processes $\widetilde Z^{\alpha,\mu}$ and $\widetilde U^{\alpha,\mu}$ satisfying (ii) in Definition \ref{def:Q} such that
\begin{align*} M_t^{\alpha,\mu}&=\widetilde Y_0(\alpha,\mu,\xi)+\int_0^t \widetilde Z_s^{\alpha,\mu}\sigma(s,X_s) \d W_s^{\alpha,\mu} +\int_0^t  \int_{\zeta\neq 0} \widetilde U^{\alpha,\mu}_s(\zeta)\widetilde N^{\alpha,\mu}(\d s,\d \zeta),\\
&=v_0(\mu,\xi)+\int_0^t \widetilde Z_s^{\alpha,\mu}\sigma(s,X_s) \d W_s^{\alpha,\mu} +\int_0^t \int_{\zeta\neq 0} \widetilde U^{\alpha,\mu}_s (\zeta) \widetilde N^{\alpha,\mu}(\d s,\d \zeta),
\end{align*}

\noindent \textit{Step 2. Change of variable and It\^o's decomposition.}
Note that $V_t(\mu,\xi)$ is a $\mathbb P^{\alpha,\mu}-$a.s. negative process as a consequence of $(\mathcal I_\xi)$ and $(\mathcal I_-)$. We define a new process $Y$ by
\[V_t(\mu,\xi)=-e^{-\gamma Y_t},\; t\in [0,T]. \]
By applying It\^o's formula, we get

\begin{align*}
Y_t&=Y_0+\int_0^t Z_s\sigma(s,X_s) d W_s+ \int_0^t \int_{\zeta\neq 0} U_s(\zeta) N(\d t,\d \zeta)-I_t-A_t,
\end{align*}
where\footnote{Note that $Z$ and $U$ are independent of $\alpha,\mu$ since they can be respectively expressed with $\langle Y, X\rangle$ and $[Y,N]$ and $A$ is the pure predictable pure jumps of $Y$. }
\[Z_t:=-\frac{\widetilde Z^{\alpha,\mu}_{t-}}{\gamma \widetilde Y_{t-}(\alpha,\mu,\xi)} ,\quad U_t(\zeta):=-\frac1{\gamma } \log\Big(1+\frac{\widetilde U^{\alpha,\mu}_t(\zeta) }{\widetilde Y_{t-}(\alpha,\mu,\xi)}\Big)- c_1(t,X_t,\zeta,\mu_t,\alpha_t) ,\]
\begin{equation*}
  \begin{aligned}
  I_t =&\int_0^t \Big[\int_{\zeta\neq 0} \frac1{\gamma }\Big(1-e^{-\gamma  (U_s(\zeta)+c_1(s,X_s,\zeta,\mu_s,\alpha_s))}\Big) K(s,X_s,\zeta,\mu_s,\alpha_s)\lambda_s^0(\d \zeta)\d s 
  \\
  &- \frac1{\gamma \widetilde Y_{t-}(\alpha,\mu,\xi)}\d A^{\alpha,\mu ; c}_s+\Big(b(s,X_s,\mu_s,\alpha_s)Z_s-\frac12\gamma\sigma^2(s,X_s)Z_s^2-c_0(s,X_s,\mu_s,\alpha_s)\Big)\d s\Big]\\
  &= \int_0^t \big[h(s,X_s,Z_s,U_s,\mu_s,\alpha_s) \d s- \frac1{\gamma \widetilde Y_{t-}(\alpha,\mu,\xi)}\d A^{\alpha,\mu ; c}_s \big]
  \end{aligned}
\end{equation*}

and
\[A_t=\frac1\gamma \sum_{s\leq t}\log\Big(1-\frac{\Delta A^{\alpha,\mu; d}_s}{\widetilde Y_{t-}(\alpha,\mu,\xi)} \Big).\]

\noindent We know that $v_T(\mu,\xi)=-1$. Then,
 \[
 0=\sup_{\alpha}\mathbb E^{\alpha,\mu}[\widetilde Y_T(\alpha,\mu,\xi)] - v_0(\xi)=\sup_{\alpha}\mathbb E^{\alpha,\mu}[\widetilde Y_T(\alpha,\mu,\xi) -M_T^{\alpha,\mu}].
\]

\noindent Since 
\begin{align*}& \mathbb E^{\alpha,\mu}[\widetilde Y_T(\alpha,\mu,\xi) -M_T^{\alpha,\mu}]\\
&=\gamma\mathbb E^{\alpha,\mu}\Big[\int_0^T\widetilde Y_{t-}(\alpha,\mu,\xi)\Big( \d I_t- h\big(t,X_t,Z_t,U_t,\mu_t,\alpha_t\big)\d t+\frac{\d A_t^{\alpha,\mu; d}}{\gamma \widetilde Y_{t-}(\alpha,\mu,\xi)}\Big)\Big],
\end{align*}
we get
\[0=\sup_{\alpha}\mathbb E^{\alpha,\mu}\Big[\int_0^T\widetilde Y_{t-}(\alpha,\mu,\xi)\Big( \d I_t-  h\big(t,X_t,Z_t,U_t,\mu_t,\alpha_t\big)\d t+\frac{\d A_t^{\alpha,\mu; d}}{\gamma \widetilde Y_{t-}(\alpha,\mu,\xi)}\Big)\Big].\]
Note that $\tilde Y$ is negative, $\d I_t - h\big(t,X_t,Z_t,U_t,\mu_t,\alpha_t\big)\d t \geq 0$ and $dA_t^{\alpha,\mu; d}$ is nonnegative. Therefore

\begin{equation*}
I_t = \int_0^t H(s,X_s,Z_s,U_s,\mu_s)\d s,\text{ and } A_t^{\alpha,\mu;d}=0. \end{equation*}

\noindent \textit{Step 3. Admissibility of $Z$ and $U$.} We now prove the integrability result $(Z,U)\in \mathcal Q$.

Note that for $\varepsilon>0$ 
\begin{align*}
&\sup_{(\alpha,\mu)}\mathbb E^{(\alpha,\mu)}[|\widetilde Y_T(\alpha,\mu,\xi)|^{1+\varepsilon}]
\\
=& \sup_{\alpha,\mu}\mathbb E^{\alpha,\mu}\big[|v_T(\mu,\xi)|^{1+\varepsilon} e^{-\gamma (1+\varepsilon) (-\int_0^T c_0(s,X_s,\mu_s,\alpha_s) ds+\int_0^T\int_{\zeta\neq 0}c_1(s,X_s,\zeta,\mu_s,\alpha_s)N(\d s,\d \zeta))}\big]\\
=& \sup_{(\alpha,\mu)}\mathbb E^{\alpha,\mu}\big[e^{-\gamma (1+\varepsilon)\xi} e^{-\gamma (1+\varepsilon) (-\int_0^T c_0(s,X_s,\mu_s,\alpha_s) ds+\int_0^T\int_{\zeta\neq 0}c_1(s,X_s,\zeta,\mu_s,\alpha_s)N(\d s,\d \zeta))}\big].\end{align*}

By Young inequality, for any $p>1$ there exists a constant $C>0$ such that
\begin{align*}&\sup_{(\alpha,\mu)}\mathbb E^{(\alpha,\mu)}[|\widetilde Y_T(\alpha,\mu,\xi)|^{1+\varepsilon}]\\&\leq C\Big( \sup_{(\alpha,\mu)}\mathbb E^{\alpha,\mu}[e^{-\gamma p (1+\varepsilon)\xi}] +   \sup_{(\alpha,\mu)}\mathbb E^{\alpha,\mu}\big[e^{-\gamma (1+\varepsilon)\tilde p (-\int_0^T c_0(s,X_s,\mu_s,\alpha_s) ds+\int_0^T\int_{\zeta\neq 0}c_1(s,X_s,\zeta,\mu_s,\alpha_s)N(\d s,\d \zeta))}\big]\Big).
\end{align*}
We choose $p$ such that $1<p\leq \frac{\gamma'}{\gamma (1+\varepsilon)}$,  for $\varepsilon\leq \tilde{\varepsilon}:= \frac{\gamma'-\gamma}{\gamma}>0$. Hence,  $\sup_{(\alpha,\mu)}\mathbb E^{\alpha,\mu}[e^{-\gamma p (1+\varepsilon)\xi}] <\infty$ because of  $(\mathcal I_\xi)$. By using $(\mathcal I_\alpha)$ we know that for any $\varepsilon\in [0,\varepsilon^*]$ and
\[ \kappa \geq \frac{\gamma'}{\gamma'-\gamma(1+\varepsilon)} (1+\varepsilon)= \frac{\frac{\gamma'}{\gamma(1+\varepsilon)}}{\frac{\gamma'}{\gamma(1+\varepsilon)}-1} (1+\varepsilon),\] 
where $\varepsilon^*=\min(\tilde \varepsilon,\frac{\kappa(\gamma'-\gamma)-\gamma'}{\gamma'+\kappa\gamma}).$
Since $1< p \leq \frac{\gamma'}{\gamma (1+\varepsilon)}$, we have

\[\kappa >\frac{\frac{\gamma'}{\gamma(1+\varepsilon)}}{\frac{\gamma'}{\gamma(1+\varepsilon)}-1} (1+\varepsilon)\geq  \frac{p}{p-1} (1+\varepsilon)= \tilde p(1+\varepsilon).\]
Then, we get from $(\mathcal I_\alpha)$,
\[  \sup_{(\alpha,\mu)}\mathbb E^{\alpha,\mu}\big[e^{-\gamma (1+\varepsilon)\tilde p (-\int_0^T c_0(s,X_s,\mu_s,\alpha_s) ds+\int_0^T\int_{\zeta\neq 0}c_1(s,X_s,\zeta,\mu_s,\alpha_s)N(\d s,\d \zeta))}\big]<\infty.\]

Therefore, for any $\varepsilon\in [0,\varepsilon^*]$ we have $\sup_{(\alpha,\mu)}\mathbb E^{\alpha,\mu}[|\widetilde Y_T(\alpha,\mu,\xi)|^{1+\varepsilon}]<\infty.$ Since $\widetilde Y(\alpha,\mu,\xi)$ is a negative super martingale, we deduce from Doob's Inequality that 

\begin{equation}\label{ytilde:varepsilon} \mathbb E^{\alpha,\mu}[\sup_{t\in [0,T]}|\widetilde Y_t(\alpha,\mu,\xi)|^{1+\varepsilon}] \leq C_\varepsilon \mathbb E^{\alpha,\mu}[|\widetilde Y_T(\alpha,\mu,\xi)|^{1+\varepsilon}] <+\infty. \end{equation}
We now recall that 
\[ -e^{-\gamma Y_t} = V_t(\mu,\xi)= \widetilde Y_t(\alpha,\mu,\xi)e^{\gamma (-\int_0^t c_0(s,X_s,\mu_s,\alpha_s) ds+\int_0^t\int_{\zeta\neq 0}c_1(s,X_s,\zeta,\mu_s,\alpha_s)N(\d s,\d \zeta))}.\]

Let $\varepsilon<\varepsilon^*$ and $q_\varepsilon=1+\varepsilon$. From H\"older Inequality with \[\rho_\varepsilon=\frac{1+\varepsilon^*}{1+\varepsilon},\; \tilde\rho_\varepsilon =\frac{\rho_\varepsilon}{\rho_\varepsilon-1}= \frac{1+\varepsilon^*}{\varepsilon^*-\varepsilon},\] together with \eqref{ytilde:varepsilon} and $(\mathcal I_\alpha)$(ii) we get

\begin{align*}&\mathbb E^{\alpha,\mu}[\sup_{t\in [0,T]}e^{-\gamma q_\varepsilon Y_t}]\\= &\mathbb E^{\alpha,\mu}[\sup_{t\in [0,T]} |\widetilde Y_t(\alpha,\mu,\xi)|^{q_\varepsilon}e^{\gamma q_\varepsilon \big(-\int_0^t c_0(s,X_s,\mu_s,\alpha_s) ds+\int_0^t\int_{\zeta\neq 0}c_1(s,X_s,\zeta,\mu_s,\alpha_s)N(\d s,\d \zeta)\big)}]\\
\leq &\mathbb E^{\alpha,\mu}[\sup_{t\in [0,T]} |\widetilde Y_t(\alpha,\mu,\xi)|^{1+\varepsilon^*}]^{1/\rho_\varepsilon}\; \mathbb E^{\alpha,\mu}\big[\sup_{t\in [0,T]} e^{\gamma \tilde\rho_\varepsilon (-\int_0^t c_0(s,X_s,\mu_s,\alpha_s) ds+\int_0^t\int_{\zeta\neq 0}c_1(s,X_s,\zeta,\mu_s,\alpha_s)N(\d s,\d \zeta))}\big]^{1/\tilde \rho_\varepsilon}\\
&<\infty. \end{align*}
Thus, $(Z,U)\in \mathcal Q$.\\

\textit{Step 4. Verification and optimality.} We verify that (IC) is satisfied for $\alpha^\star=\hat \alpha$ maximizer of $H$ when $\xi=Y_T^{Y_0,Z,U}$. For any $(\mu,\alpha)\in\mathcal P_2^1\times \widetilde{\mathcal A}^\mu$ we set
\begin{align*}
\overline Y^{Y_0,Z,U}_t:= Y_t^{Y_0,Z,U}-\int_0^t c_0(s,X_s,\mu_s,\alpha_s) \d s+\int_0^t \int_{\zeta\neq 0}c_1(s,X_s,\zeta,\mu_s,\alpha_s)N(\d s,\d \zeta).
\end{align*}
By It\^o's formula we have
\begin{align*}&\;\;\;\d e^{-\gamma  \overline Y^{Y_0,Z,U}_t}
\\
&= \gamma e^{-\gamma  \overline Y^{Y_0,Z,U}_t} \Big(-Z_t\sigma(t,X_t)\d W^{\alpha,\mu}_t -\int_{\zeta\neq 0} \frac1{\gamma }\Big(1-e^{-\gamma  (u(\zeta)+c_1(t,X_t,\zeta,\mu_t,\alpha_t))}\Big)\tilde N^{\alpha,\mu}(\d s,\d \zeta) \Big) \\
&+ \gamma e^{-\gamma  \overline Y^{Y_0,Z,U}_t} (H(t,X_t,Z_t,U_t(\cdot),\mu_t)-h(t,X_t,Z_t,U_t(\cdot),\mu_t,\alpha_t)) \d t.\end{align*}
By definition of $H$, we know that $e^{-\gamma \overline Y^{Y_0,Z,U}}$ is a $\mathbb P^{\alpha,\mu}-$ local submartingale of class $(D)$ by using $(\mathcal I_\xi),(\mathcal I_\alpha)$ and H\"older Inequality. Hence,  $e^{-\gamma \overline Y^{Y_0,Z,U}}$ is a $\mathbb P^{\alpha,\mu}-$ submartingale. We deduce that 
\begin{align*}
v_0(\alpha;\mu,\xi)&=\mathbb E^{\alpha,\mu}[-e^{-\gamma \overline Y^{Y_0,Z,U}_T}]\\
&= - e^{-\gamma Y_0} -\mathbb E^{\alpha,\mu}\Big[\int_0^T  \gamma e^{-\gamma  \overline Y^{Y_0,Z,U}_t} (H(t,X_t,Z_t,U_t,\mu_t)-h(t,X_t,Z_t,U_t,\mu_t,\alpha_t)) \d t\Big]\\
&\leq -e^{-\gamma Y_0},
\end{align*}
and equality holds if and only if $\alpha$ is a maximizer of $H$.\\

\textit{Step 5. Uniqueness of the representation.} We finally prove that the uniqueness of the triplet $(Y_0,Z,U)$ associated with a contract $\xi\in \mathcal C$ such that $\xi=Y_T^{Y_0,Z,U}$. Assume that there exist $(Y_0,Z,U)\in \mathbb R\times \mathcal Q$ and $(Y_0',Z',U')\in \mathbb R\times \mathcal Q$ such that $\xi=Y_T^{Y_0,Z,U}=Y_T^{Y_0',Z',U'}.$ Therefore, we deduce from the previous step that $e^{-\gamma \overline Y^{Y_0,Z,U}}$ and $e^{-\gamma \overline Y^{Y_0',Z',U'}}$ are two martingales with the same terminal value at time $t=T$. Hence, $Y_t^{Y_0,Z,U}=Y_t^{Y_0',Z',U'}$ for any time $t\leq T$. Consequently, $Y_0=Y_0'$, $Z=Z'$ and $U=U'$.
\end{proof}

{Before turning to the solution to the Stackelberg mean field game \ref{bilevel}, we deduce from the previous theorem a fundamental corollary to characterize any $(\alpha^\star,\mu^\star)\in \text{MFG}(\xi)$ as a solution of a system of controlled McKean-Vlasov stochastic differential equations. We introduce below this system with solution $(X,Y)$ defined respectively in \eqref{sde:controlled} and Definition \ref{def:Q} (ii).}

\begin{equation*}
\text{MKV}(Z,U)\begin{cases}
&\d X_t=  b(t,X_t,\mu^\star_t,\alpha^\star_t)\d t+\sigma(t,X_t)\d W^{\alpha^\star,\mu^\star}_t + \int_{\zeta\neq 0} \zeta N(\d t,\d \zeta),\\
 & X_0=x,
  \\[0.8em]
& \d Y_t^{Y_0,Z,U}  =\big(b(t,X_t,\mu^\star_t,\alpha^\star_t)Z_t-H(t,x,Z_t,U_t,\mu_t) \big)\d t +Z_t\sigma(t,X_t) d W_t^{\alpha^\star,\mu^\star} 
\\
&\qquad\qquad\quad+\int_{\zeta\neq 0} U_t(\zeta)  N(\d t,\d \zeta),\\[0.5em]
 & Y_0^{Y_0,Z,U}=Y_0,\\[0.8em]
& \alpha^\star_t=\hat{\alpha}(t,X_t,Z_t,U_t, \mu^\star_t),\quad \mu^\star_t=\mathbb P^{\alpha^\star,\mu^\star}\circ X_t^{-1},
  \end{cases}
\end{equation*}
The following corollary states the equivalence of MFG$(\xi)$ and MKV$(Z,U)$.

\begin{corollary}[(MFG)$(\xi)$ characterization]\label{mfgcharacterization}
\[\mathcal C \equiv \Xi:=\{\xi=Y_T^{Y_0,Z,U} \text{ where $Y^{Y_0,Z,U}$ is the solution to \text{MKV}(Z,U) with } (Z,U)\in\mathcal Q,\; Y
_0\geq \hat R_0 \},\]
with $\hat R_0=-\frac{\log(-R_0)}{\gamma}$.
In other words, any mean field equilibrium $(\alpha^\star,\mu^\star)\in\text{(MFG)}(\xi)$ is related to the solution $(X,Y)$ of a controlled McKean-Vlasov SDE driven by a L\'evy process with 
\[\alpha^\star_t:=\hat\alpha(t,X_t,Z_t,U_t,\mu^\star)\in \hat{\mathcal A}(t,X_t,Z_t,U_t, \mu^\star),\quad \mu^\star_t=\mathbb P^{\alpha^\star,\mu^\star}\circ X_t^{-1}. \]

\end{corollary}

\subsection{Mean field PDE}\label{section:MFPDE}

\noindent The previous results in Corollary \ref{mfgcharacterization} enable us to rewrite the problem of the principal as follow:

\begin{align*}
&V^P_0=\sup_{(Y_0,Z,U,\alpha^\star,\mu^\star)\in [\hat R_0,+\infty)\times \mathcal Z\times \mathcal U \times \text{MFG}(\widehat Y^{Y_0,Z,U}_T)}\; \mathbb E^{\alpha^\star,\mu^\star}\Big[ U_P\Big(\pi(X_T) -\widehat Y_T^{Y_0,Z,U}\Big)\Big],\\[0.5em]
\nonumber &\text{subjected to }\text{($\mathcal I_\xi$), ($\mathcal I_{\alpha^\star}$), $(R)$},
\end{align*}
with

\begin{equation*}
 \widehat{\text{MKV}}(Z,U) \begin{cases}
&\d X_t=  b(t,X_t,\mu^\star_t,\alpha^\star_t)\d t+\sigma(t,X_t)\d W^{\alpha^\star,\mu^\star}_t + \int_{\zeta\neq 0} \zeta N(\d t,\d \zeta),\\
 & X_0=x,
  \\[0.8em]
& \d \widehat Y_t  =\big(b(t,X_t,\mu^\star_t,\alpha^\star_t)Z_t-\hat H(t,X_t,Z_t,U_t,\mu_t) \big)\d t +Z_t\sigma(t,X_t) d W_t^{\alpha^\star,\mu^\star} \\ 
  &\qquad\qquad\quad+\int_{\zeta\neq 0} \big(U_t(\zeta)  +\ell(t,X_t,\mu_t^\star)\big)N(\d t,\d \zeta),\\[0.5em]
 & \widehat Y_0^{Y_0,Z,U}=Y_0,\\[0.8em]
& \alpha^\star_t=\hat{\alpha}(t,X_t,Z_t,U_t, \mu^\star),\quad \mu^\star_t=\mathbb P^{\alpha^\star,\mu^\star}\circ X_t^{-1},
  \end{cases}
\end{equation*}
where $\hat H(t,x,z,u,\mu)=H(t,x,z,u,\mu)-g(t,x)$.
Note that the reservation utility constraint (R) is saturated, so that the problem is reduced to 
\begin{align}\label{bilevel:reduced}
&V^P_0=\sup_{(Z,U,\alpha^\star,\mu^\star)\in  \mathcal Z\times \mathcal U \times \text{MFG}(\widehat Y^{\hat R_0,Z,U}_T)}\; \mathbb E^{\alpha^\star,\mu^\star}\Big[ U_P\Big(\pi(X_T) -\widehat Y_T^{\hat R_0,Z,U}\Big)\Big],\\[0.5em]
\nonumber &\text{subjected to }\text{($\mathcal I_\xi$), ($\mathcal I_{\alpha^\star}$).}
\end{align}
The optimization problem \eqref{bilevel:reduced} corresponds to a stochastic control problem of the system of controlled McKean-Vlasov SDEs with jumps $\widehat{\text{MKV}}(Z,U)$. See \cite{guo2020s,burzoni2020viscosity}.\\

We introduce the operator $\L^{z,u,\mu}_t$ on the set of twice continuously differentiable function on $\R^2$ defined as
\begin{equation*}
  \begin{aligned}
  &\L^{z,u,\alpha,\mu}_t v (x,y)\\
  & =b\big(t,\alpha,x,\mu_t\big)\big(\partial_x v(x,y)+z\partial_y v(x,y)\big) -\hat H(t,z,u,x,\mu)\partial_y v(x,y)\\
  &+\int_{\zeta\neq 0}\Big(v\big(x+\zeta,y+u_t(\zeta)+\ell(t,x,\mu)\big)-v(x,y)\Big)K(t,x,\zeta,\mu,\alpha)\lambda_t^0(\d \zeta)
  \\
  &+\frac12 \sigma^2(t,x)\big(\partial_{xx}v (x,y)+2z\partial_{xy}v (x,y)+z^2\partial_{yy}v (x,y)\big) ,
  \end{aligned}
\end{equation*}
for any $(t,x,y,z,u,\mu,\alpha)\in [0,T]\times \mathbb R\times \mathbb R\times \mathbb R\times \mathcal G(\mathbb R)\times \mathcal P_2(\mathbb R)\times \mathbb R^d$. The corresponding HJB equation is 
\begin{equation}\label{HJB}
  \begin{cases}
  -\partial_t V(t,\rho) - \sup_{z,u} \big\langle \rho,\sup_{\alpha\in \hat{\mathcal A}(t,x,z,u,\rho^x)}\L^{z,u,\alpha,\rho^x}_t D_m V\big\rangle=0,\quad (t,\rho)\in [0,T)\times \mathcal P_2(\mathbb R^2)
  \\
  V(T,\rho)=\int_{\R^2} U_P(\pi(x)-y)\rho(\d x\d y),
  \end{cases}
\end{equation}
where $\rho^x$ denotes the marginal of $\rho$ with respect to $X$. 

\begin{theorem}[Verification and optimal contract]\label{thm:verif}
Let $V$ be a continuous map from $[0, T] \times \P_2^2([0,T])$ into $\R$ such that $v(t, \cdot)$ is twice continuously differentiable on $\P_2^2([0,T])$ and such that $V(\cdot, \rho)$ is continuously differentiable on $[0, T]$.  Suppose that $V$ is solution to \eqref{HJB} satisfying
\begin{equation*}
\sup_{t\in [0,T],\rho\in \mathcal{K}}\bigg[\int_{\R^2}\Big(|\nabla_xD_mV(t,\rho,x)|^2+|\nabla^2_{x}D_mV(t,\rho,x)|^2\Big)\mu(\d x) \bigg]<\infty,
\end{equation*}
for any compact set $\mathcal{K}\in \P_2^2([0,T])$. Moreover, the supremum is attained for some optimizers denoted by $Z^\star_t,U^\star_t,\alpha^\star_t$ for any $(t, \rho)\in [0, T] \times \P_2^2([0,T])$ such that $(t, \rho) \to Z^\star_t,U^\star_t$ and $(t,\rho,x)\to \alpha^\star_t$  are continuous.
Let $(X_0,Y_0)$ be a square random variable with law $\rho_0\in\P_2^2([0,T])$ and assume moreover that the following McKean–Vlasov SDE
\begin{equation*}
 \begin{cases}
&\d X_t=  b\big(t,X_t,\rho_t^x,\alpha^\star_t(\rho_t,D_mV,X_t)\big)\d t+\sigma(t,X_t)\d W^{\alpha^\star,\rho^x}_t + \int_{\zeta\neq 0} \zeta N(\d t,\d \zeta),
  \\[0.8em]
& \d \widehat Y_t  =\big(b\big(t,X_t,\rho_t^x,\alpha^\star_t(\rho_t,D_mV,X_t)\big)Z^\star_t(\rho_t,D_mV)-\hat H\big(t,X_t,Z^\star_t(\rho_t,D_mV),U^\star_t(\rho_t,D_mV),\rho^x_t\big) \big)\d t 
\\
  &+Z^\star_t(\rho_t,D_mV)\sigma(t,X_t) d W_t^{\alpha^\star,\rho^x} +\int_{\zeta\neq 0} \big(U^\star_t(\rho_t,D_mV,\zeta)  +\ell(t,X_t,\rho^x_t)\big)N(\d t,\d \zeta),\\[0.8em]
& \rho_t=\mathbb P^{\alpha^\star,\rho^x}\circ (X_t,Y_t)^{-1},\;\rho^{x}_t = \mathbb P^{\alpha^\star,\rho^x}\circ (X_t)^{-1}.
  \end{cases}
\end{equation*} 
admits a solution $(X,Y)$. Then if the tuple of control $(Z^\star_t,U^\star_t)\in\mathcal Q$ and if $\alpha_t^\star\in\hat A(t,X_t,Z_t^\star,U_t^\star,\rho^x)$ we have $V(0, \rho_0) = V_0^P$ and $Z^\star,U^\star,\alpha^\star$ are optimal in the problem of the Principal.

\end{theorem}

\begin{proof}
The proof is a slight modification of Theorem 4.1 in \cite{guo2020ito} using their Theorem 3.3. We notice that as an optimizer, $\alpha^\star_t$ serves a different role and is determined continuously given any $Z^\star_t$, $U^\star_t$ prior to solving the HJB equation. Now the only difference would be the dimension and the fact that $Z^\star_t$ and $U^\star_t$ only depend on the law, not on the state. In this case, the Ito's formula of the flow of measure still holds due to a similar argument of Theorem 3.3 in their paper under two dimensions and the rest of the proof follows.
\end{proof}

\section{Example: energy demand-response pricing under blackout or sudden shortages}\label{section:electricity}
In this section, we study an explicit model where the dynamic of the output process depends on its mean and its variance and L\'evy process is a purely discontinuous process with controlled intensity and jump size one. The example proposed is motivated by  \cite{aid2022optimal,elie2021mean} to random powercut or shortage occurrences. This framework is greatly inspired by the current energy sobriety policy in Europe either for ecological motivations\footnote{See for example \url{https://www.eceee.org/all-news/news/energy-sobriety-a-disruptive-notion-catching-on-in-france/} or \url{https://www.pge.com/en_US/residential/save-energy-money/savings-programs/savings-programs-overview/savings-programs-overview.page}} or under energy blackout threatening. 

\subsection{The model and the value functions} We consider the following output process representing either the profit and loss of the energy producer or the difference between the baseline consumption of a crowd of similar consumers and their current consumption:  \begin{equation*}
  \d X_t = \big(\alpha^0_t + k_1\E[X_t]\big) \d t + \sigma \d W_t - \d J_t,
\end{equation*}
where
\begin{itemize}
\item $\alpha^0$ is the control made by consumers to reduce the distance between their consumption and their usual baseline, note that $\alpha^0$ can be positive or negative. When $X\geq 0$, the consumers are under their current baseline. This case is wished by the energy producer caring about energy sobriety. 
\item $k_1\in \mathbb R$ is a synergistic parameter,
\item $J_t$ represent the accident occurrence at time $t$ with frequency $e^{-k_2\alpha^1_t}\text{Var}(X_t)$ for $k_2>0$. We assume that this frequency of blackout/shortage can be reduced by consumers' efforts $\alpha^1$ on the one hand, but is very sensitive with respect to the variability of energy consumption on the other hand. It enables us to extend the main stylized facts introduced in \cite{aid2022optimal,elie2021mean} to accidents occurring when the network is subjected to variability either endogenous (excessive strains) or exogenous (political tensions with exporting energy country).  
\end{itemize}

The value function for a crowd of agents (that is a crowd of similar consumers)
\begin{equation*}
  v_0(\alpha,\mu,\xi) = \E^{{\alpha,\mu}} \bigg[U_A\Big(\xi-\int_0^T\Big(\frac{|\alpha^0_t|^2}{2}-f(t,X_t)\Big)\d t 
  + \int_0^T  \alpha^1_t \d J_t\Big)\bigg],
\end{equation*}
where $f:[0,T]\times \mathbb R\longrightarrow \mathbb R$ is decreasing with respect to the $X$ variable so that the consumers wins utility by increasing their energy consumption. The value function for the Principal (the energy producer) is given by
\begin{equation*}
  V_0^P=\sup\E^{\alpha,\mu}\bigg[\beta X_T-\xi-\int_0^Tg(t,X_t)\d t -\frac{\theta}{2}\text{Var}(X_T)\bigg],
\end{equation*}
such that $g:[0,T]\times \mathbb R\longrightarrow \mathbb R$ is decreasing with respect to the $X$ variable so that the producer wins utility when $X$ increases, fitting with energy sobriety policies; $\beta,\theta\in \mathbb R^+$. As in \cite{aid2022optimal,elie2021mean}, we assume in addition that $f$ and $g$ are continuously differentiable and $f(t,x)-g(t,x)=\delta x$ for some $\delta\in \mathbb R$.\\

\subsection{HJB equation on the space of measures and optimal tarification}
Translating to our previous setting in Section \ref{section:model}, we have 
{\begin{equation}\label{example:parameters}
\begin{aligned}
&\sigma(t,x)=\sigma,\; \alpha=(\alpha^0,\alpha^1),\; b(t,x,\mu,\alpha) = \alpha^0 + k_1m(\mu),\\
 &\ell=0,\quad c_0(t,x,\alpha,\mu) =\frac{|\alpha^0|^2}{2}-f(t,x),\quad c_1(t,x,z,\alpha,\mu)=\alpha^1,\quad 
  \\
  &\lambda_t^0(\d \zeta) = \delta_{-1}(\d \zeta), \qquad K(t,x,z,\mu,\alpha) = e^{-k_2\alpha^1}\text{Var}(\mu),
  \end{aligned}
\end{equation} }
where $m(\mu)=\int \tilde x\mu (\d \tilde x)$ and $\text{Var}(\mu)=\int \tilde x^2\mu (\d \tilde x)-\big(\int \tilde x\mu (\d \tilde x)\big)^2$ and $\delta_{-1}$ represent the Dirac measure at $-1$.\\

Hence, by using the same notations that those in Section \ref{section:MFG}, we have \begin{equation*}
  \begin{aligned}
  h(t,x,\alpha,z,u,\mu)=&f(t,x)+z\alpha^0-\frac{|\alpha^0|^2}{2}+z k_1m(\mu)-\frac{\gamma z^2\sigma^2}{2}+\frac{1-e^{-\gamma(u(-1)+\alpha^1)}}{\gamma}e^{-k_2\alpha^1}\text{Var}(\mu),
  \end{aligned}
\end{equation*}
Therefore, the optimal effort of the system of similar agents is $\alpha^\star=(  \alpha^{0,\star},   \alpha^{1,\star})$ with
\begin{equation*}
  \alpha^{0,\star}=z,\;\;\alpha^{1,\star}=\frac1\gamma\log\big(1+\frac\gamma {k_2}\big)-u(-1),
\end{equation*}
and consequently, 
\begin{equation*}
  \begin{aligned}
  \hat H(t,x,z,u,\mu)=&\delta x+\frac{1}{2}z^2-\frac{\gamma\sigma^2}{2}z^2+z k_1m(\mu)+\frac{1}{\gamma+k_2}\Big(\frac{\gamma+k_2}{k_2}\Big)^{-\frac {k_2}\gamma}e^{k_2u(-1)}\text{Var}(\mu),
  \end{aligned}
\end{equation*}
where $\hat H$ is defined in Section \ref{section:MFPDE}. 
Furthermore, the operator $\L$ can be simplified as follow
\begin{equation*}
  \begin{aligned}
  \L^{z,u,\mu}_t v (x,y) = &b\big(t,\alpha^\star,x,\mu\big)\big(\partial_x v(x,y)+z\partial_y v(x,y)\big) -\hat H(t,x,z,u,\mu)\partial_y v(x,y)
  \\
  &+\frac12 \sigma^2(t,x)\big(\partial_{xx}v (x,y)+2z\partial_{xy}v (x,y)+z^2\partial_{yy}v (x,y)\big)
  \\
  &+K\big(t,x,z,\mu,\alpha^\star\big) \big(v(x-1,y+u(-1))-v(x,y)\big).
  \end{aligned}
\end{equation*}
We aim to find an explicit solution of \eqref{HJB}. We guess that $V(t,\mu)$ has the following form:
\begin{equation*}
  \begin{aligned}
  V(t,\mu)=&h_0(t)+h_1(t)m(\mu_x)+h_2(t)\text{Var}(\mu_x)-m(\mu_y),
  \end{aligned}
\end{equation*}
for some function $h_0,h_1,h_2$ to be determined. In this case, we have
\begin{equation*}
  \begin{aligned}
  D_mV(t,\mu,x,y) = &h_1(t)x+h_2(t)\big(x^2-2m(\mu_x)x\big)-y.
  \end{aligned}
\end{equation*}
Set $v(x,y):=D_mV(t,\mu,x,y)$, then 
\begin{equation*}
\partial_x v(x,y) = h_1(t)+2h_2(t)\big(x-m(\mu_x)\big),\;\;\partial_{xx}v=2h_2(t),\;\;\partial_y v(x,y)= -1,
\end{equation*}
and $v(x-1,y+u)-v(x,y)=-\partial_x v+ u(0) \cdot\partial_y v+\partial_{xx} v / 2$. We denote $b^\star=b(t,x,\alpha^\star,\mu)=z+k_1m(\mu)$ and $K^\star=K(t,x,z,\mu,\alpha^\star)=\big(\frac{\gamma+k_2}{k_2}\big)^{-k_2/\gamma}e^{k_2u(-1)}\text{Var}(\mu_x)$, then
\begin{equation*}
  \begin{aligned}
  \L^{z,u(\cdot),\mu_x}_tD_m V = &\,\partial_x v \cdot (b^\star-K^\star) + \partial_y v (b^\star z+K^\star u(-1)-H)+h_2(t)\lambda^\star+\frac{1}{2}\partial_{xx}v\sigma^2
   \\
  =&\,z\partial_x v-\frac{z^2}2(1+\gamma\sigma^2) +k_1\partial_x vm(\mu_x)+h_2(t)\sigma^2+\delta x
  \\
  &+ \Big(\frac{\gamma+k_2}{k_2}\Big)^{-\frac {k_2}\gamma}\Big(\frac{1}{\gamma+k_2}-u-\partial_x v+h_2(t)\Big)e^{k_2u(-1)}\text{Var}(\mu_x).
  \end{aligned}
\end{equation*}
Then, we have
\begin{equation*}
  \begin{aligned}
  \langle\mu,\L^{z,u(\cdot),\mu_x}_tD_m V\rangle =& z\langle\mu,\partial_x v\rangle-\frac{z^2}2(1+\gamma \sigma^2) + \big(k_1\langle \mu, \partial_x v\rangle+\delta\big) m(\mu_x)+h_2(t)\sigma^2  
  \\
  &+ \Big(\frac{\gamma+k_2}{k_2}\Big)^{-\frac {k_2}\gamma}\Big(\frac{1}{\gamma+k_2}+u(0)-\langle\mu,\partial_x v\rangle+h_2(t)\Big)e^{k_2u(-1)}\text{Var}(\mu_x),
  \end{aligned}
\end{equation*}
To maximize $\langle\mu,\L^{z,u,\mu_x}_tD_m V\rangle $ over $z,u$, we have 
\begin{equation*}
z^\star=\frac{\langle\mu,\partial_x v\rangle}{1+\gamma \sigma^2}, \qquad u^\star(-1)=\frac{1}{\gamma+k_2}-\langle\mu,\partial_x v\rangle+h_2(t)-\frac{1}{k_2}.
\end{equation*}
Therefore, 
\begin{equation*}
  \begin{aligned}
     \sup_{z,u}\langle\mu,\L^{z,u,\mu_x}_tD_m V\rangle =&\frac{\langle\mu,\partial_x v\rangle^2}{2(1+\gamma \sigma^2)}+\big(k_1\langle \mu, \partial_x v\rangle+\delta\big) m(\mu_x)+h_2(t)\sigma^2
     \\
     &+\frac{1}{k_2}\Big(\frac{\gamma+k_2}{k_2}\Big)^{-\frac {k_2}\gamma}e^{-\frac{\gamma}{\gamma+k_2}-k_2\langle\mu,\partial_x v\rangle+k_2h_2(t)}\text{Var}(\mu_x).
  \end{aligned}
\end{equation*}
Recall the definition of $\partial_x v$, we have $\langle\mu,\partial_x v\rangle=h_1(t)$, plugging into the previous term, we have
\begin{equation*}
  \begin{aligned}
    \sup_{z,u}\langle\mu,\L^{z,u,\mu_x}_tD_m V\rangle = &\frac{h_1(t)^2}{2(1+\gamma\sigma^2)}+h_2(t)\sigma^2+\big(h_1(t)k_1+\delta\big)m(\mu_x)
    \\
    &+\frac{1}{k_2}\Big(\frac{\gamma+k_2}{k_2}\Big)^{-\frac {k_2}\gamma}e^{-\frac{\gamma}{\gamma+k_2}-k_2h_1(t)+k_2h_2(t)}\text{Var}(\mu_x).
  \end{aligned}
\end{equation*}
To prove that $-\partial_t V(t,\mu) - \sup_{z,u} \big\langle \mu,\L^{z,u,\mu_x}_t D_m V\big\rangle=0$, we have to solve
\begin{equation*}
(\text{ODE})\begin{cases}
  h_2'(t)+\frac{1}{k_2}\big(\frac{\gamma+k_2}{k_2}\big)^{-k_2/\gamma}e^{-\frac{\gamma}{\gamma+k_2}-k_2h_1(t)+k_2h_2(t)}=0,\;\;h_2(T)=-\frac{\theta}{2},
  \\
  h_1'(t)+h_1(t)k_1+\delta=0,\;\;h_1(T)=\beta,
  \\
  h_0'(t)+\frac{h_1(t)^2}{2(1+\gamma\sigma^2)}+h_2(t)\sigma^2=0, \;\;h_0(T)=0.
\end{cases}
\end{equation*}
The ODE for $h_1(t)$ in (ODE) can be solved explicitly and we get $h_1(t)=\big(\beta+\frac{\delta}{k_1}\big)e^{k_1(T-t)}-\frac{\delta}{k_1}$. Now plugging back to $h_2(t)$, we have
\begin{equation*}
  {\big(e^{-k_2 h_2(t)}\big)} '=\Big(\frac{\gamma+k_2}{k_2}\Big)^{-k_2/\gamma}e^{-\frac{\gamma}{\gamma+k_2}-k_2h_1(t)},
\end{equation*}
Therefore 
\begin{equation*}
  \begin{aligned}
  e^{-k_2h_2(t)} =& e^{\frac{k_2\theta}{2}}-\Big(\frac{\gamma+k_2}{k_2}\Big)^{-k_2/\gamma}e^{-\frac{\gamma}{\gamma+k_2}+\frac{k_1\delta}{k_2}}\int_0^{T-t} e^{-(\beta+\frac{\delta}{k_1})k_2e^{k_1s}}\d s
  \end{aligned}
\end{equation*}
For small time interval $T$, the right hand side is strictly larger than $0$, and we get 
\begin{equation*}
  h_2(t) = -\frac{1}{k_2}\ln\Big(e^{\frac{k_2\theta}{2}}-\Big(\frac{\gamma+k_2}{k_2}\Big)^{-k_2/\gamma}e^{-\frac{\gamma}{\gamma+k_2}+\frac{k_1\delta}{k_2}}\int_0^{T-t} e^{-(\beta+\frac{\delta}{k_1})k_2e^{k_1s}}\d s\Big),
\end{equation*}
and 
\begin{equation*}
  h_0(t) = \int_t^T \Big(\frac{h_1(s)^2}{2(1+\gamma\sigma^2)}+h_2(s)\sigma^2\Big)\d s.
\end{equation*} 

Consequently, the optimal incentive parameters $Z^\star$ and $U^\star$ are given by

\begin{equation}\label{optimalcontrol}
Z^\star_t=\frac{h_1(t)}{1+\gamma \sigma^2}, \qquad U_t^*(-1)=\frac{1}{\gamma+k_2}-h_1(t)+h_2(t)-\frac{1}{k_2}.
\end{equation}
Therefore, the optimal effort of the consumers are
\begin{equation}\label{optimalcontrol:agent}
\alpha_t^{0,\star}=\frac{h_1(t)}{1+\gamma \sigma^2},\;\;\alpha_t^{1,\star}=\frac1\gamma\log\big(1+\frac\gamma {k_2}\big)-\frac{1}{\gamma+k_2}+h_1(t)-h_2(t)+\frac{1}{k_2}.
\end{equation}
\subsection{Admissibility of the proposed optimal policy} We now have to check that the optimal controls given in \eqref{optimalcontrol} satisfied the conditions $(\mathcal I_\alpha),(\mathcal I_-),(\mathcal I_{\xi})$ and $(\text{MF})$. It is clear that $(\mathcal I_\alpha),(\mathcal I_-),(\mathcal I_{\xi})$ hold since $Z^\star$ and $U^\star$ are bounded. We focus on the existence of a mean field equilibrium for these optimal controls.
Recall that
\begin{equation*}
  \begin{aligned}
  b^\star&=\alpha^{1,\star}_t+k_1m(\mu_x)=Z_t^\star+k_1m(\mu_x)=\frac{h_1(t)}{1+\gamma\sigma^2}+k_1m(\mu_x),\\
  K^\star&=\Big(\frac{\gamma+k_2}{k_2}\Big)^{-\frac {k_2}{\gamma}}e^{-k_2U_t^\star(-1)}\text{Var}(\mu_x)=\Big(\frac{\gamma+k_2}{k_2}\Big)^{-\frac {k_2}\gamma}e^{-\frac{\gamma}{\gamma+k_2}-k_2h_1(t)+k_2h_2(t)}\text{Var}(\mu_x).
  \end{aligned}
\end{equation*}
Putting back into the diffusion of $X_t$, we have
\begin{equation}\label{X:opt}
  \d X_t = b^\star\d t + \sigma \d W_t^{\mu,\alpha^\star} -\d J_t
\end{equation}
Taking expectation, we get
\begin{equation*}
  \begin{aligned}
  &\frac{\d m(\mu_{X_t})}{\d t} = \frac{h_1(t)}{1+\gamma\sigma^2}+k_1m(\mu_{X_t})-h(t)\text{Var}(\mu_{X_t}),
  \end{aligned}
\end{equation*}
with 
\begin{equation*}
  h(t):=\Big(\frac{\gamma+k_2}{k_2}\Big)^{-\frac {k_2}\gamma}e^{-\frac{\gamma}{\gamma+k_2}-k_2h_1(t)+k_2h_2(t)}.
\end{equation*}

Using Ito's formula, we have
\begin{equation*}
  \d |X_t|^2 = (2X_{t-}b^\star+\sigma^2)\d t + \d W_t^{\mu,\alpha^\star} - (2X_{t-}+1)\d J_t,
\end{equation*}
Taking expectation, and we denote $q(\mu)=\int \tilde x ^2\mu (\d \tilde x)$, we get
\begin{equation*}
  \frac{\d q(\mu_{X_t})}{\d t} = 2k_1 |m(\mu_{X_t})|^2 +\sigma^2+\frac{2h_1(t)}{1+\gamma\sigma^2}m(\mu_{X_t}) - h(t)(2m(\mu_{X_t})+1)\text{Var}(\mu_{X_t})
\end{equation*}
To conclude, we have the following ODE system of $(m(t),q(t))=(m(\mu_{X_t}),q(\mu_{X_t}))$:
\begin{equation}\label{ode:moment}
 \begin{cases}
  m'(t) = \frac{h_1(t)}{1+\gamma\sigma^2}+k_1m(t)-h(t)\big(q(t)-|m(t)|^2\big),
  \\
  q'(t) = 2k_1 |m(t)|^2 +\sigma^2+\frac{2h_1(t)}{1+\gamma\sigma^2}m(t) - h(t)(2m(t)+1)\big(q(t)-|m(t)|^2\big),
  \end{cases}
\end{equation}
To simplify, let us denote $v(t):=q(t)-|m(t)|^2$ as the variance process, then we have
\begin{equation*}
  \begin{cases}
  m'(t) = k_1m(t)-h(t)v(t)+\frac{h_1(t)}{1+\gamma\sigma^2},
  \\
  v'(t) = h(t)v(t)+\sigma^2.
  \end{cases}
\end{equation*}
and existence and uniqueness holds for all time since it is a linear ODE system. Since there is a unique optimizer $\alpha^\star$ with deterministic optimal control $Z^\star$ and $U^\star$ and a unique solution to (ODE), we deduce from Corollary \ref{mfgcharacterization} the existence of a unique mean field equilibrium, so that $\alpha^\star,Z^\star,U^\star$ satisfy (MF).  

\subsection{Optimal incentive policy}
Following all the previous computations, the solution to the problem in this example is given by the following proposition.
\begin{proposition}\label{prop:example}
The compensation $\xi^\star=Y_T^{\hat R_0,Z^\star,U^\star}$ with parameters $(Z^\star,U^\star)$ defined by \eqref{optimalcontrol} and the pair $(\alpha^\star,\mu^\star)$ given by \eqref{optimalcontrol:agent} with $\mu^\star$ being the law of $X$ defined by \eqref{X:opt} with fixed point first and second moment solving \eqref{ode:moment} are the unique optimizers of the bilevel programming \eqref{bilevel} {with parameters given by \eqref{example:parameters}}. 
\end{proposition}

\newcommand{\etalchar}[1]{$^{#1}$}
\providecommand{\bysame}{\leavevmode\hbox to3em{\hrulefill}\thinspace}
\providecommand{\MR}{\relax\ifhmode\unskip\space\fi MR }
\providecommand{\MRhref}[2]{%
  \href{http://www.ams.org/mathscinet-getitem?mr=#1}{#2}
}
\providecommand{\href}[2]{#2}

\end{document}